\documentclass[a4paper,twoside]{article}

\usepackage[utf8]{inputenc}
\usepackage{amsmath,amsthm}
\usepackage{amssymb,amsfonts}
\usepackage{graphicx}
\usepackage{url}
\usepackage{hyperref}
\usepackage{color}
\usepackage{xcolor}
\usepackage{framed}
\usepackage[font=small]{caption}
\colorlet{shadecolor}{orange!15}
\definecolor{Red}{rgb}{0.8,0,0}

\usepackage{lastpage}

\usepackage{fancyhdr}
\newcommand{\N}{\mathbb{N}}

\newcommand{\R}{\mathbb{R}}

\newtheorem{theorem}{\color{Red}\underline{\textrm Theorem}}
\newenvironment{theo}
  {\begin{shaded}\begin{theorem}}
  {\end{theorem}\end{shaded}}
\newtheorem{lemma}{Lemma}

\author{ Jacopo D'Aurizio \\ {\small Università di Pisa} \\ {\small elianto84@gmail.com}}
\title{\vspace{-2cm}\bfseries A generalization of the Shafer-Fink inequality}
\date{02-03-2013}

\begin{document}
\maketitle
\thispagestyle{empty}
\noindent In this article we will prove some generalizations and extensions of the Shafer-Fink (\cite{Shafer}) double inequality 
for the arctangent function:
\begin{theo}
For any positive real number $x$,
$$ \frac{3 x}{1+2\sqrt{1+x^2}} < \arctan x < \frac{\pi x}{1+2\sqrt{1+x^2}}$$ 
holds.
\end{theo}
\begin{proof}
Following the lines of (\cite{QiZhangGuo}), we consider the substitution $x=\tan\theta$, that gives the following, equivalent form of the inequality:
$$ \forall \theta\in I=(0,\pi/2),\qquad  \theta(\cos\theta + 2)-\pi \sin\theta < 0 < \theta(\cos\theta + 2)-3 \sin\theta. $$
If now we set
$$f_K(\theta) = (\cos\theta + 2)- K\,\frac{ \sin\theta}{\theta}$$
we have:
$$ \theta^2\,\frac{df_K}{d\theta}=(K-\theta^2)\sin\theta-K\theta\cos\theta.$$
Since for any $\theta\in I$ we have:
$$ \frac{\theta}{\tan \theta}< 1-\frac{\theta^2}{3}< 1-\frac{\theta^2}{\pi}, $$
$f_3(\theta)$ ed $f_\pi(\theta)$ are both non-decreasing on $I$, in virtue of $\frac{df_K}{d\theta}\geq 0$;
moreover, $f_K'(0)=0$ and $f_K'$ cannot be zero on $I$. Since:
$$ f_3(0)=0, \quad f_3(\pi/2)>0,\quad f_\pi(0)<0,\quad f_\pi(\pi/2)=0, $$
the claim follows.
\end{proof}

$$\phantom{}$$
\noindent We give now a different proof of this inequality, that relies on the bisection formula for the cotangent function and the associated Weierstrass product.
$$\phantom{}$$

\noindent From the logarithmic derivative of the Weierstrass product for the sine function we know that for any $x\in[0,\pi/2]$
$$ f(x) = x\cot x = 1-2\sum_{k=1}^{+\infty}\frac{\zeta(2k)}{\pi^{2k}}\,x^{2k} $$
holds. Since $f(x)$ is an even function, there exists a suitable linear combination $g_1(x)$ of $f(x)$ and $f(x/2)$ that satisfies:
$$ g_1(x) = A_0 f(x) + A_1 f(x/2) = 1-\sum_{k\geq 2} C^{(1)}_k\, x^{2k}. $$
With the choices $A_0 = -\frac{1}{3}, A_1=\frac{4}{3}$ the previous identity holds, and, for any $k\geq 2$:
$$ C^{(1)}_k = \left(A_0  + \frac{A_1}{4^k}\right)\frac{\zeta(2k)}{\pi^{2k}} < 0, $$
so $g_1(x)$ is an increasing and convex function over $I=[0,\pi/2]$. From that,
$$ \forall x\in I,\quad  \left(-\frac{1}{3} x\cot x + \frac{2}{3} x\cot\frac{x}{2}\right) \in [g_1(0),g_1(\pi/2)]= [1,\pi/3]$$
follows. If now we consider the bisection formula for the cotangent function:
$$ \cot\frac{x}{2} = \cot x + \sqrt{1+\cot^2 x} $$
we have a different proof of the Shafer-Fink inequality.\\
$$\phantom{}$$

\noindent We consider now $g_2(x)$ as a linear combination of $f(x),f(x/2)$ and $f(x/4)$ such that:
$$ g_2(x) = A_0 f(x) + A_1 f(x/2) + A_2 f(x/4) = 1-\sum_{k\geq 3} C^{(2)}_k\, x^{2k}. $$
From the annihilation of the coefficient of $x^2$ in the RHS we deduce the constraint $A_0+A_1\cdot\frac{1}{4}+A_2\cdot\frac{1}{16}=0$, 
and from the annihilation of the coefficient of $x^4$ we deduce the constraint $A_0+A_1\cdot\frac{1}{16}+A_2\cdot\frac{1}{256}=0$. 
If we take $p_2(x)=A_0 + A_1 x + A_2 x^2$, such constraints translate into $p_2(1/4) = p_2(1/16) = 0$, from which:
$$ p_2(x) = K_2\left(x-\frac{1}{4}\right)\left(x-\frac{1}{16}\right), $$
with $K_2=(1-1/4)^{-1}\cdot (1-1/16)^{-1}$ in order to grant $A_0 + A_1 + A_2 = p_2(1) = 1$.\\
Since $C^{(2)}_k = \frac{\zeta(2k)}{\pi^{2k}}p_2(4^{-k})$, all the non-zero coefficients of the Taylor series of $g_2(x)$, except (at most) the first one,
have the same sign, so $g_2(x)$ is a monotonic function over $I$. In particular: 

\begin{align*} 
\forall x\in I,\qquad     \frac{\pi(3+8\sqrt{2})}{45}=g_2(\pi/2) \leq g_2(x) &= \frac{1}{45}\left(f(x)-20 f(x/2) + 64 f(x/4)\right) \\[+0.2cm]
                                                &= \frac{x}{45}\left( \cot x - 10 \cot(x/2) + 16 \cot(x/4)\right) \leq 1,
\end{align*}
from which we get:
$$ \pi(3+8\sqrt{2}) \leq x \left( \cot x - 10 \cot(x/2) + 16 \cot(x/4)\right)\leq 45. $$
By using twice the bisection formula for the cotangent, we have the following strengthening of the Shafer-Fink inequality:
\begin{theo}[D'Aurizio]
For any positive real number $x$
$$ \pi(3+8\sqrt{2})\cdot f(x) < \arctan x < 45\cdot f(x) $$
holds, where:
$$ f(x) = \frac{x}{7+6\,\sqrt{1+x^2}+16\sqrt{2}\,\sqrt{1+x^2+\sqrt{1+x^2}}}. $$
\end{theo}

\noindent The same approach leads to an arbitrary strengthening of the Shafer-Fink inequality:

\begin{theo}[D'Aurizio]\label{master}
For any positive real number $x$ and for any positive natural number $n$,\\ once defined:
$$ f(x) = x\cot x = 1-2\sum_{k=1}^{+\infty}\frac{\zeta(2k)}{\pi^{2k}}\,x^{2k}, $$
$$ p_n(x) = \prod_{k=1}^{n}\frac{(4^k x-1)}{(4^k-1)} = A_0 + A_1 x +\ldots + A_n x^n, $$
$$ g_n(x) = \sum_{k=0}^{n} A_k\, f(2^{-k} x) = x \sum_{k=0}^{n} \frac{A_k}{2^k}\, \cot(2^{-k} x), $$
$$ e_j(x_1,\ldots,x_k) = \sum_{sym}x_1\cdot\ldots\cdot x_j, $$
$$ L_0(x)=1, \qquad L_{n+1}(x) = L_n(x) + \sqrt{x^2+L_n(x)^2}, $$ 
we have:
$$  K_{low}\cdot a_n(x) < \arctan(x) < K_{high}\cdot a_n(x), $$
where $K_{low}=\min(g_n(0),g_n(\pi/2))$, $K_{high}=\max(g_n(0),g_n(\pi/2))$ and:
$$ a_n(x) = x\cdot\left(\sum_{j=0}^n (-1)^{n-j}\cdot L_j(x)\cdot 2^j\cdot e_j(1,4,\ldots,4^{n-1})\right)^{-1}.$$
Moreover, $K_{high}-K_{low} < \frac{1}{4^n}$.
\end{theo}
\begin{proof}
By taking
$$ p_n(x) = \prod_{k=1}^{n}\frac{(4^k x-1)}{(4^k-1)} = A_0 + A_1 x +\ldots + A_n x^n $$
we have $p_n(1)=1$ and $p_n(4^{-j})=0$ for every $j\in[1,n]$. In particular, the Taylor series of
$$ g_n(x) = \sum_{k=0}^{n} A_k\, f(2^{-k} x) = x \sum_{k=0}^{n} \frac{A_k}{2^k}\, \cot(2^{-k} x). $$
is equal to:
$$ 1-2\sum_{k=1}^{+\infty}\frac{\zeta(2k)p_n(4^{-k})}{\pi^{2k}}\,x^{2k} = 1-2\sum_{k>n}C^{(k)}_n\,x^{2k}, $$
and all the $C^{(k)}_n$ with $k>n$ have the same sign, so $g_n(x)$ is monotonic over $[0,\pi/2]$, with $g_n(0)=1$.\\
In particular, we have:
$$\forall x\in[0,\pi/2],\qquad x\cdot\sum_{j=0}^n (-1)^{n-j}\cot\left(\frac{x}{2^j}\right)2^j\,e_j(1,4,\ldots,4^{n-1}) \leq \prod_{k=1}^{n}(4^k-1), $$
where $e_j$ is the $j$-th elementary symmetric function. Since for any $m>n$ we have $|p_n(4^{-m})|<1$, 
$$ \left|g_n(\pi/2)-g_n(0)\right| \leq \sum_{k>n}\frac{\zeta(2k)}{4^k} < \frac{1}{4^n}. $$
holds.
\end{proof}

We give now another upper bound for the arctangent function that does not belong to the last family of inequalities, 
but that strenghtens the inequality $\arctan x < \frac{\pi x}{1+2\sqrt{1+x^2}}$, too.

\begin{theo}
For any positive real number $x$
$$ \arctan x < \frac{\pi x}{\frac{4}{\pi}+\sqrt{2}\sqrt{1+x^2+x\sqrt{1+x^2}}}$$ 
holds.
\end{theo}

\begin{proof}
By using the substitution $x=\tan\theta$, it is sufficient to prove that for any $\theta\in I=[0,\pi/2]$ we have:
$$ \theta \leq \frac{\pi \sin\theta}{\frac{4}{\pi}\cos\theta +\sqrt{2+2\sin\theta}},$$ 
that is also equivalent, up to the change of variable $\theta=\pi/2-\phi$, to the inequality:
$$ \frac{\pi}{2}-\phi \leq \frac{\pi \cos\phi}{\frac{4}{\pi}\sin\phi + 2\cos(\phi/2)},$$ 
or the inequality:
$$ \frac{\cos\phi}{1-\frac{2\phi}{\pi}}\geq \cos(\phi/2)\left(\frac{4}{\pi}\sin(\phi/2)+1\right).$$
In order to prove the latter it is sufficient to prove:
$$ \frac{\cos\phi}{1-\frac{2\phi}{\pi}}\geq \cos(\phi/2)\left(1+\frac{2\phi}{\pi}\right),$$ 
or:
$$ \frac{\cos\phi}{1-\frac{4\phi^2}{\pi^2}}\geq \cos(\phi/2).$$ 
By considering the Weierstrass product for the cosine function we may rewrite the last line in the form:
$$ \prod_{k=1}^{+\infty}\left(1-\frac{4x^2}{(2k+1)^2 \pi^2}\right) \geq \prod_{k=1}^{+\infty}\left(1-\frac{x^2}{(2k-1)^2 \pi^2}\right). $$
By considering the Taylor series of the logarithm of both sides, we simply have to prove:
$$\forall m\in\N_0,\qquad (4^m-1)\zeta(2m)-4^m-(1-4^{-m})\zeta(2m) \leq 0,$$
that is a consequence of:
$$\forall m\in\N_0,\qquad \zeta(2m)\leq\frac{4^m+1}{4^m-1}, $$
implied by:
$$\forall m\in\N_0,\qquad (4^m-1)(\zeta(2m)-1) \leq 2. $$
An upper bound for the LHS is the series:
$$ 1+\sum_{k=1}^{+\infty}\left(\frac{4}{(2k+1)^2}\right)^m, $$
whose value decreases as $m$ increases; so we have:
$$(4^m-1)(\zeta(2m)-1) \leq  1+\sum_{k=1}^{+\infty}\frac{4}{(2k+1)^2} = 3\zeta(2)-3,$$
and the RHS is less than $2$ since $\pi^2<10$ holds.
\end{proof}
$$\phantom{}$$
Now we make a step back into the general setting of double inequalities for the arctangent function.
\begin{lemma}
If $f(u),g(u)$ are a couple of real functions such that, for any $u\in[0,1]$, 
$$ f(u) \leq \arctan u \leq g(u) $$
holds, then:
$$ 2\cdot f\left(\frac{x}{1+\sqrt{1+x^2}}\right) \leq \arctan x \leq 2\cdot g\left(\frac{x}{1+\sqrt{1+x^2}}\right) $$
holds for any $x\in\R^+$.
\end{lemma}
\begin{proof}
In virtue of the angle bisector theorem,
$$ \arctan t = 2 \arctan\left(\frac{t}{1+\sqrt{1+t^2}}\right) $$
for any $t\geq 0$, so if the first inequality holds for any $\theta=\arctan u$ in the range $[0,\pi/4]$, the second inequality
holds for any $\theta=\arctan x$ in the range $[0,\pi/2]$.
\end{proof}
The last lemma gives a third way to prove the Shafer-Fink inequality. By direct inspection of the Taylor series of $\frac{\arctan u}{u}$, 
it is easy to show that $(3+u^2)\frac{\arctan u}{u}$ is an increasing function over $[0,1]$, so:
$$ \frac{3u}{3+u^2}\leq \arctan u\leq\frac{\pi u}{3+u^2}, $$
and it is sufficient to use the substitution $u=\frac{x}{1+\sqrt{1+x^2}}$ to give another proof of the Shafer-Fink inequality.

\begin{lemma}
If an approximation $f(u)$ of the arctangent function satisfies:
$$ \| f(u)-\arctan(u) \|_{\R^+} = \sup_{u\in\R^+}|f(u)-\arctan(u)| = C_\infty, $$
then
$$ \left\| 2\cdot f\left(\frac{u}{1+\sqrt{1+u^2}}\right)-\arctan(u)\right\|_{\R^+} = 2\cdot \| f(u)-\arctan(u) \|_{(0,1)} = 2\cdot C_1, $$
and, for any $t\in(0,1)$,
$$ \left\| 2\cdot f\left(\frac{u}{1+\sqrt{1+u^2}}\right)-\arctan(u)\right\|_{(0,t)} = 2\cdot \| f(u)-\arctan(u) \|_{\left(0,\frac{2t}{1-t^2}\right)}.$$
\end{lemma}
This simple consequence of the previous lemma tell us the fact that any algebraic approximation of the arctangent function 
in a right neighbourhood of zero can be ``lifted'' to an algebraic approximation over the whole $\R^+$, through the iteration of the map
$$ f(u)\quad\longrightarrow\quad 2\cdot f\left(\frac{u}{1+\sqrt{1+u^2}}\right). $$
For example, if we consider the Lagrange interpolation polynomial for the arctangent function with respect to the points $(0,\tan(\pi/8)=\sqrt{2}-1,\tan(\pi/4)=1)$
$$ p(x) = \frac{\pi}{4}\cdot\frac{x(x-\sqrt{2}+1)}{2-\sqrt{2}}+\frac{\pi}{8}\cdot\frac{x(x-1)}{(\sqrt{2}-1)(\sqrt{2}-2)}, $$
we have
$$ \|p(x)-\arctan x\|_{(0,1)}<\frac{1}{230}, $$
so, by considering $ 2\cdot p\left(\frac{x}{1+\sqrt{1+x^2}}\right)$:
\begin{theo}
For any non negative real number $x$, the absolute difference between $\arctan(x)$ and
$$\frac{\pi  x \left(\left(4+\sqrt{2}\right) \left(1+\sqrt{1+x^2}\right)-\sqrt{2}\,x\right)}{8 \left(1+\sqrt{1+x^2}\right)^2} $$ 
is less than $\frac{1}{115}$.
\end{theo}

Another way to produce really effective approximation is to use the Chebyshev expansion for the arctangent function:
\begin{lemma}
The sequence of functions:
$$ f_n(x) = 2\sum_{k=0}^n \frac{(-1)^k}{(2k+1)(1+\sqrt{2})^{2k+1}}\;T_{2k+1}(x), $$
where $T_k(x)$ is the $k$-th Chebyshev polynomial of the first kind, gives a uniform approximation of the arctangent function over the interval $[0,1]$:
$$ \| \arctan x - f_n(x) \|_{[0,1]}\leq \frac{1}{(1+\sqrt{2})^{2n+3}}.$$
Moreover,
$$ \arctan(m x) = 2\sum_{k=0}^{+\infty} \frac{(-1)^k}{(2k+1)}\left(\frac{m}{1+\sqrt{1+m^2}}\right)^{2k+1}\,T_{2k+1}(x) $$
holds for any $x\in (-1,1)$ and for any $m\in\N_0$.
\end{lemma}

\begin{theo}
For any $n\in\N_0$ and for any $x\in\R$
$$ \left|\;\arctan x - 4\sum_{k=0}^n \frac{(-1)^k}{(2k+1)(1+\sqrt{2})^{2k+1}}\;T_{2k+1}\left(\frac{x}{1+\sqrt{1+x^2}}\right)\right| \leq \frac{1}{\left(3+2\sqrt{2}\right)^n}.$$
\end{theo}

\noindent Still another way is to use the continued fraction representation for the arctangent funtion:
$$ \arctan z = \frac{z}{1+\frac{z^2}{3+\frac{4z^2}{5+\frac{9z^2}{7+\frac{16z^2}{9+\frac{25z^2}{11+\ldots}}}}}}, $$
from which we get a sequence of approximations for $\arctan x$ over $[0,1]$:
$$\left\{\begin{array}{cll} K_1(x) &=\displaystyle\frac{x}{1+x^2/3},&\\[+0.4cm]
                            K_2(x) &=\displaystyle\frac{x}{1+x^2/(3+4x^2/5)}&=\displaystyle\frac{x(15+4x^2)}{15+9x^2},\\[+0.4cm]
                            K_3(x) &=\displaystyle\frac{x}{1+x^2/(3+4x^2/(5+9x^2/7))}&=\displaystyle\frac{5 x \left(21+11 x^2\right)}{105+90 x^2+9 x^4}\\[+0.4cm]
                            \ldots & & \end{array}\right.$$
that satisfy:
$$ \|\arctan x - K_n(x) \|_{[0,1]} \leq \frac{1}{2\cdot 4^n}, $$
so:
$$ \left\|\arctan x - K_n\left(\frac{x}{1+\sqrt{1+x^2}}\right) \right\|_{\R} \leq \frac{1}{4^n}, $$
with an error term that is the same achieved by $a_n(x)$, defined as in Theorem (\ref{master}).



$$\phantom{}$$
By using the Taylor series for the arctangent function with respect to the point $x=1$ one has:
$$ \arctan{x} = \frac{\pi}{4}-\sum_{j=0}^{+\infty}\left(-\frac{(1-x)^4}{4}\right)^j\cdot\left(\frac{(1-x)}{2(4j+1)}+\frac{(1-x)^2}{2(4j+2)}+\frac{(1-x)^3}{4(4j+3)}\right).$$
By plugging in $x=2/3$ we have:
$$ \arctan\frac{1}{5} = \sum_{j=0}^{+\infty}\left(-\frac{1}{324}\right)^j\cdot\left(\frac{1}{6(4j+1)}+\frac{1}{18(4j+2)}+\frac{1}{108(4j+3)}\right), $$
and by plugging in $x=119/120$ we have:
$$ \arctan\frac{1}{239} = \sum_{j=0}^{+\infty}\left(-\frac{1}{829440000}\right)^j\cdot\left(\frac{1}{240(4j+1)}+\frac{1}{28800(4j+2)}+\frac{1}{6912000(4j+3)}\right). $$
The Machin Formula
$$ \frac{\pi}{4} = 4\arctan\frac{1}{5}+\arctan\frac{1}{239} $$
give us the possibility to exhibit a good approximation for $\pi$:
\begin{align*}
 \pi=8\;&\sum_{j=0}^{+\infty}\left(-\frac{1}{324}\right)^j\cdot\left(\frac{1}{3(4j+1)}+\frac{1}{9(4j+2)}+\frac{1}{54(4j+3)}\right)+\\
  +&\sum_{j=0}^{+\infty}\left(-\frac{1}{829440000}\right)^j\cdot\left(\frac{1}{60(4j+1)}+\frac{1}{7200(4j+2)}+\frac{1}{1728000(4j+3)}\right).
\end{align*}
In the same fashion, we have that: 
$$ \arctan\frac{1}{2z-1} = \sum_{j=0}^{+\infty}\left(-\frac{1}{4z^4}\right)^j\cdot\left(\frac{1}{2z(4j+1)}+\frac{1}{2z^2(4j+2)}+\frac{1}{4z^3(4j+3)}\right)$$
holds for any $z\geq 1$, and the truncated series gives a better and better approximation as $z$ goes to infinity. By a change of variable, the same is true for:
$$ \arctan\frac{1}{t} = \sum_{j=0}^{+\infty}\left(-\frac{4}{(t+1)^4}\right)^j\cdot\left(\frac{1}{(t+1)(4j+1)}+\frac{2}{(t+1)^2(4j+2)}+\frac{2}{(t+1)^3(4j+3)}\right),$$
and:
$$ \arctan u = \sum_{j=0}^{+\infty}\left(-\frac{4u^4}{(u+1)^4}\right)^j\cdot\left(\frac{u}{(u+1)(4j+1)}+\frac{2u^2}{(u+1)^2(4j+2)}+\frac{2u^3}{(u+1)^3(4j+3)}\right)$$
holds for any $u\in[0,1]$. By taking:
$$ s_n(u) =  \sum_{j=0}^{n}\left(-\frac{4u^4}{(u+1)^4}\right)^j\cdot\left(\frac{u}{(u+1)(4j+1)}+\frac{2u^2}{(u+1)^2(4j+2)}+\frac{2u^3}{(u+1)^3(4j+3)}\right)$$
we have that:
$$ \left| \arctan u - s_n(u) \right| \leq \left(\frac{\sqrt{2}\,u}{u+1}\right)^{4n} $$
for any $u\in[0,1]$, with $s_n$ being an upper bound for $\arctan u$ over $[0,1]$ for any even $n$ and a lower bound for any odd $n$. If we consider:
$$ t_n(u) = \frac{\pi}{4}-s_n\left(\frac{1-u}{1+u}\right) = \frac{\pi}{4}-\sum_{j=0}^{n}\left(-\frac{(1-u)^4}{4}\right)^j\cdot\left(\frac{1-u}{2(4j+1)}+\frac{(1-u)^2}{2(4j+2)}+\frac{(1-u)^3}{4(4j+3)}\right),$$
then $t_n$ is a lower/upper bound for the arctangent function over $[0,1]$ if and only if $s_n$ is a lower/upper bound, and:
$$ \left| \arctan u - t_n(u) \right| \leq \left(\frac{1-u}{\sqrt{2}}\right)^{4n} $$ 
holds. Any convex combination of $s_n$ and $t_n$ is still a lower/upper bound - by taking:
$$ w_n(u) = \frac{u^{4n+4}\cdot t_n(u) + (1-u)^{4n+4}\cdot s_n(u)}{u^{4n+4}+(1-u)^{4n+4}} $$
we can perform a reduction of the uniform error, since:
$$ \left| w_n(u) - \arctan u \right| \leq \frac{1}{20^n} $$
and the error function goes very fast to zero when $u$ approaches $0$ or $1$. This gives that
$$ w_n\left(\frac{u}{1+\sqrt{1+u^2}}\right) $$
is an especially good lower/upper bound for the arctangent function when $u$ is close to $0$ or much bigger than $1$, achieving the same uniform error term 
with respect to the generalized Shafer-Fink inequality or the continued fraction expansion.


\end{document}